\documentclass[12pt]{article}
\usepackage{amssymb}
\usepackage{a4wide}
\usepackage{epsfig}
\usepackage{oldgerm}
\usepackage{amsmath}
\usepackage{cite}
\usepackage{amsthm}
\usepackage{fixmath}
\usepackage{multicol}
\usepackage{amstext,color}
\usepackage{array,multirow}
\usepackage{fullpage}
\usepackage{graphicx,graphics}
\usepackage{caption}
\usepackage{subcaption}

\usepackage[vlined,ruled,linesnumbered]{algorithm2e}

\definecolor{teal}{rgb}{0.0,0.5,0.5}
\definecolor{frenchrose}{rgb}{0.96,0.29,0.54}
\definecolor{lg}{rgb}{0.36,0.99,0.82}
\definecolor{dblue}{rgb}{0,0,.5}
\definecolor{dpink}{cmyk}{.2,1,.1,.04}
\definecolor{purple}{rgb}{0.35,0.04,0.64}
\definecolor{borange}{rgb}{1, .388, 0}
\definecolor{dpurple}{rgb}{0.61,0.22,1.00}
\definecolor{purp}{rgb}{0.44,0.00,0.87}
\definecolor{green}{rgb}{0.00,0.44,0.00}
\definecolor{junebud}{rgb}{0.74,0.85,0.34}
\definecolor{plum}{rgb}{0.56,0.27,0.52}
\newtheorem{theorem}{Theorem}[section]

\newtheorem{definition}{Definition}[section]
\newtheorem{lemma}{Lemma}[section]
\newtheorem{remark}{Remark}[section]

\newtheorem{example}{Example}
\def\CC{{\textmd \kern.24em \vrule width.02em height1.4ex depth-.05ex \kern-.26emC}}
\def\TagOnRight
\def\QQ{\rlap {\raise 0.4ex \hbox{$\scriptscriptstyle |$}} {\hskip -0.1em Q}}
\catcode`\@=11
\begin{document}
\begin{center}
  {{\bf \large {\rm {\bf Symmetric fractional order reduction method with $L1$ scheme on graded mesh for time fractional nonlocal diffusion-wave equation of Kirchhoff type }}}}
\end{center}
\begin{center}
	{\textmd {\bf Pari J. Kundaliya,}}\footnote{\it Department of Mathematics,  Institute of Infrastructure, Technology, Research and Management, Ahmedabad, Gujarat, India, (pariben.kundaliya.pm@iitram.ac.in)}
	{\textmd {{\bf Sudhakar Chaudhary}}}\footnote{\it Department of Mathematics,  Institute of Infrastructure, Technology, Research and Management, Ahmedabad, Gujarat, India, (sudhakarchaudhary@iitram.ac.in)}
\end{center}
\begin{abstract}
  In this article, we propose a linearized fully-discrete scheme for solving a time fractional nonlocal diffusion-wave equation of Kirchhoff type. The scheme is established by using the finite element method in space and the $L1$ scheme in time. We derive the $\alpha$-robust \textit{a priori} bound and \textit{a priori} error estimate for the fully-discrete solution in $L^{\infty}\big(H^1_0(\Omega)\big)$ norm, where $\alpha \in (1,2)$ is the order of time fractional derivative. Finally,  we perform some numerical experiments to verify the theoretical results.
\end{abstract}
{\bf Keywords:} Time fractional diffusion-wave equation; Nonlocal problem; Weak singularity; $L1$ scheme; $\alpha$-robust; Convergence estimate  \\
\noindent {\bf AMS(MOS):} 35R11, 65M60, 65M12.

\section{Introduction}

 In recent years, the time fractional diffusion-wave equations are studied very intensively (see \cite{[FMPP],[TSZT],[OA],[JHDL],[Ch1hyp],[sfor1],[Sty_frac_wave],[H2N2_interploation]} and references therein).  Fractional diffusion-wave equation can be derived from classical diffusion equation or wave equation by replacing the first or second time derivative with fractional order derivative \cite{[TSZT],[OA]}.   Fractional diffusion-wave equation has applications in modeling of anomalous diffusion processes, electromagnetic and acoustic responses \cite{[JHDL],[ZZXW2006]}.

 In this work, we consider the following time fractional nonlocal diffusion-wave equation of Kirchhoff type:
 \begin{subequations}\label{4e1}
 	\begin{align}
 	\label{4cuc3:1.1}
 	^{c}_{0}{D}^{\alpha}_{t}u(x,t)-a(l(u)) \: \Delta u(x,t)&=f(x,t)  \quad \mbox{in}  \quad \Omega\times(0,T], \\
 	\label{4cuc3:1.2}
 	u(x,t)&=0  \quad \mbox{on}  \quad \partial\Omega\times(0,T],\\
 	\label{4cuc3:1.3}
 	u(x,0)=u_0(x) \; \mbox{and} \; u_t(x,0)&=u_1(x) \quad \mbox{in} \quad \Omega,
 	\end{align}
 \end{subequations}
 where $\Omega \subseteq \mathbb{R}^d$ $(d=1$ or $2)$ is a bounded domain with smooth boundary $\partial \Omega,$ $u_t = \frac{\partial u}{\partial t}$ and $l(u)= \int_{\Omega} |\nabla u|^2 \, dx \, = \|\nabla u \|^2.$ The term $^{c}_{0}{D}^{\alpha}_{t}u(x,t)$ is the $\alpha ^{th}$ order Caputo fractional derivative of $u$ with $\alpha \in (1,2)$ and it is defined \cite{[r1]} as \\
 \begin{equation}\label{4e2a}
 \begin{split}
 {^{c}_{0}D^{\alpha}_{t}}u(x,t) =& \frac{1}{\Gamma (2-\alpha)}\int_{0}^{t}(t-s)^{1-\alpha} \; \frac{\partial^2 u(x,s)}{\partial s^2} \, ds , \;\; t>0.
 \end{split}
 \end{equation}
 When $\alpha =1,$ \eqref{4e1} reduces to nonlocal diffusion equation of Kirchhoff type \cite{[MCBL2003],[SZMC2005]} and when $\alpha =2,$ \eqref{4e1} reduces to nonlocal wave equation of Kirchhoff type \cite{[sv_h],[SI_P],[A_Arosio_1996],[gs_Kr]}. For $1<\alpha <2,$  equation \eqref{4e1}  is expected to interpolate nonlocal diffusion and wave equations of Kirchhoff type. This is why we call \eqref{4e1} the time fractional nonlocal diffusion-wave equation of Kirchhoff type.
 Many papers in the literature have been produced on the study of existence-uniqueness of following nonlocal wave equation of Kirchhoff type (which is corresponding to $\alpha=2$)\cite{[sv_h],[SI_P],[gs_Kr]},
 \begin{subequations}\label{4e1a}
 	\begin{align}
 	\label{4cuc3:1a.1}
 	u_{tt} - M \Big( \int_{\Omega} |\nabla u|^2 \, dx \Big) \: \Delta u &=f(x,t) \quad  \mbox{in}  \; \, \Omega\times(0,T], \\
 	\label{4cuc3:1a.2}
 	u(x,t)&=0  \quad \mbox{on}  \; \, \partial\Omega\times(0,T],\\
 	\label{4cuc3:1a.3}
 	u(x,0)=u_0(x) \; \mbox{and} \; u_t(x,0)&=u_1(x) \quad \mbox{in} \; \, \Omega.
 	\end{align}
 \end{subequations}
 Equation \eqref{4e1a} describes the vibrations of an elastic string in $1D$ or a membrane in $2D$ \cite{[uber_machanik]}. Authors in \cite{[I_Christie_1984],[J_Peradze_2005]} developed numerical algorithms for solving 1D nonlocal Kirchhoff equation \eqref{4e1a}. In \cite{[sv_h]}, finite element method with Newton iteration is used to solve 2D nonlocal Kirchhoff equation \eqref{4e1a}. In \cite{[TSZT]}, authors consider a general $1D$ time fractional wave equation for a vibrating string and discussed about its solution by using method of separation of variables and Laplace transform. \\
Getting analytical solution of \eqref{4e1} is either not possible or difficult due to nonlocal nature of time fractional derivative and presence of nonlocal diffusion coefficient. Now a days, study of numerical solution of fractional order PDEs has become a significant area of research among researchers. There are many work in the literature related to numerical analysis of
 	time fractional wave equations \cite{[Ch1hyp],[sfor1],[Sty_frac_wave],[H2N2_interploation],[distri_ord_wave_eq],[l_fdm_frac_wave],[time_multi_term_frac]}. In particular, authors in \cite{[distri_ord_wave_eq]} considered a time distributed-order wave equation and solved it by using alternating direction implicit difference scheme.
 In \cite{[l_fdm_frac_wave],[time_multi_term_frac]}, authors introduced a new variable for order reduction of time derivative and derived a temporal second-order scheme to obtain the numerical solution of a time multi-term fractional wave equation. The symmetric fractional order reduction method is used in \cite{[sfor1]} to solve a semilinear fractional diffusion-wave equation.
 The solution $u$ of a time fractional wave equation possibly has a weak singularity near $t=0$ \textit{i.e.,} $u_{tt}$ blows up at $t = 0$ \cite{[Sty_frac_wave]}. In present work, we also assume that solution of \eqref{4e1} has a weak singularity near time $t=0.$ When we use $L1$ scheme on uniform mesh, weak singularity often leads to suboptimal convergence rate \cite{[sp1],[r8]}.
 To overcome this difficulty, we use graded mesh \cite{[kwn],[Ch1hyp],[Sty_frac_wave]} to discretize the time interval.\\
 In this work, we propose a linearized scheme for solving \eqref{4e1}. This scheme is based on extrapolation \cite{[Dilip_2022]} and it uses the $L1$ scheme with graded mesh in time \cite{[sp1],[r02],[r11]} and finite element method in space. Since $L1$ scheme is developed for solving time fractional diffusion equation ($0 < \alpha < 1$), we need to transform problem \eqref{4e1} into a lower order system of equations. For this we follow the idea given in \cite{[sfor1]}. Also, the \textit{a priori} bound and error estimate which we derive in the present work are $\alpha$-robust \cite{[r15]}.
 \\
 \noindent The organization of paper is as follows: In Section 2, we provide some usual notations and reformulate problem \eqref{4e1} using symmetric fractional order reduction method. Linearized fully-discrete scheme is given in Section 3 to obtain the numerical solution of problem \eqref{4e1}.  \textit{A priori} bound and \textit{a priori} error estimate for the fully-discrete solution are derived in Section 4 and Section 5, respectively. Numerical experiments are provided in Section 6 to confirm the theoretical findings.  \\
\section{Preliminaries and Order reduction}
Notations: Throughout the article, $C > 0$ denotes a generic constant independent of mesh parameters $h$ and $N.$  We use $(\cdot,\cdot)$ for the inner product and $ \|\cdot\|$ for the norm on space $L^2(\Omega).$  For $m \in \mathbb{N}$, $H^m(\Omega)$ denotes the standard Sobolev space with the norm $\|\cdot\|_m$ and space $H_{0}^1(\Omega)$ consist of functions from $H^1(\Omega)$ whose traces vanish on the boundary $\partial\Omega.$ Also, we write $\|\cdot\|_{L^{\infty}(H^m)} = \sup_{0<t\leq T}\|\cdot\|_m, \ m\in \mathbb{N}.$\\
For our analysis, we require following hypotheses on functions $a$, $f$, $u_0$ and $u_1$ \cite{[sv_h]}:
\begin{itemize}
	\item {H1:} For the function $a: \mathbb{R} \rightarrow \mathbb{R}$, there exist $m_1, m_2 >0$ such that  $0<m_1 \le a(x) \le m_2 < \infty, \; \forall \, x \in \mathbb{R}.$
	\item {H2:} The function $a: \mathbb{R} \rightarrow \mathbb{R}$ is Lipschitz continuous. \textit{i.e.,} there exists $L>0$ such that
	\begin{equation}\label{e5}
	|a(x_1)-a(x_2)| \le L \, |x_1-x_2|, \quad \forall \, x_1, x_2 \in \mathbb{R}.
	\end{equation}
	\item {H3:} $f \in L^{2}(0, T; L^2(\Omega))$ and $u_0,$ $u_1 \in H_0^1(\Omega)\cap H^2(\Omega)$.
\end{itemize}
In this work, we also assume that problem \eqref{4e1} has a unique solution with sufficient regularity which is required in our further analysis.\\

Now, we reformulate problem \eqref{4e1} using symmetric fractional order reduction technique \cite{[sfor1]}. For this we need following Lemma.
\begin{lemma}\label{4l1}
   \cite{[sfor1]} For $\alpha \in (1,2)$ and any function $z(t) \in C^1[0,T] \cap C^2(0,T],$ it holds that
   \begin{equation}\label{4e2}
      ^{c}_{0}D^{\alpha}_t z(t) = \, ^{c}_{0}D^{\alpha}_t \overline{z}(t) = \, ^{c}_{0}D^{\frac{\alpha}{2}}_t \big( ^{c}_{0}D^{\frac{\alpha}{2}}_t \overline{z}(t) \big),
   \end{equation}
   where $\overline{z}(t) = z(t) - t \, z'(0).$
\end{lemma}
\noindent Next, we assume that $\beta = \frac{\alpha}{2},$ $\overline{u}(t) = u(t) - t \, u_1$ and $v= \, ^{c}_{0}{D}^{\beta}_{t}\overline{u}.$ Thus, $v(x,0)=0, \, \forall x \in \Omega$ \cite{[sfor1]}. Now, by applying Lemma \ref{4l1}, problem \eqref{4e1} can be transformed to following system of equations:
\begin{subequations}\label{4e3}
	\begin{align}
	\label{4cuc3:3.1}
	^{c}_{0}{D}^{\beta}_{t}v(x,t)-a(l(u)) \: \Delta \overline{u}(x,t)&=f(x,t) + t \ a(l(u)) \: \Delta u_1  \quad \mbox{in}  \quad \Omega\times(0,T], \\
	\label{4cuc3:3.2}
	^{c}_{0}{D}^{\beta}_{t}\overline{u} - v &=0 \quad \mbox{in} \quad \Omega\times(0,T],\\
	\label{4cuc3:3.3}
	\overline{u}(x,t)=0 \; \mbox{and} \; v(x,t)&=0 \quad \mbox{on}  \quad \partial\Omega\times(0,T],\\
	\label{4cuc3:3.4}
	\overline{u}(x,0)=u_0(x) \; \mbox{and} \; v(x,0)&=0 \quad \mbox{in} \quad \Omega.
	\end{align}
\end{subequations}
Now, the weak formulation of problem \eqref{4e3} is to find $\overline{u}(\cdot,t), \, v(\cdot,t) \in  H^1_0(\Omega)$ such that for each $t \in(0,T]$ and $\forall \omega \in H^1_0(\Omega),$
\begin{subequations}\label{4e4}
 \begin{align}
   \label{4e4:a}
     \big(^{c}_{0}{D}^{\beta}_{t}v, \omega \big) \, + \, a\big(l(u)\big) \, (\nabla \overline{u}, \nabla \omega) \, =& \, (f, \omega) \, + \, t \, a\big(l(u)\big) \, (\Delta u_1, \omega), \\
   \label{4e4:b}
     \big(^{c}_{0}{D}^{\beta}_{t}\overline{u}, \omega \big) - (v, \omega) \, =& \, 0, \\
   \label{4e4:c}
      \big( \overline{u}(x,0), \omega \big) = \big(u_0(x), \omega \big) \; &\mbox{and} \; \big( v(x,0), \omega \big) =0.
 \end{align}
\end{subequations}
\section{Fully-discrete scheme}
In order to derive the fully-discrete scheme, first we discretize the spatial domain $\Omega$.
%
Let $\Omega_h$ be a partition (quasi-uniform) of domain $\Omega$ into disjoint subintervals in 1D and triangles $T_{k}$ in 2D with a step size $h$. For a finite integer $M>0$, we consider the $M$-dimensional subspace $X_h$ of $H^{1}_{0}(\Omega)$ which is defined below.
$$X_h:=\Big\{w\in C^{0}(\bar{\Omega}): w_{|{\small T_{k}}}\in P_1(T_k), \: \forall \: T_{k}\in \Omega_h \: \, \mbox{and} \: \, w=0 \; \mbox{on} \; \partial \Omega\Big\}.$$
 Further, we define the function $\phi _i \in X_h$ (for each $i= 1,2,...,M$) in such a way that $\phi _i$ takes the value 1 at $i^{th}$ node and vanishes at other node points. Clearly, $\left\lbrace \phi _i  \right\rbrace ^{M}_{i=1}$ forms a basis for $X_h.$\\
In the following we recall the definition of discrete \textit{Laplacian}, $L^2$ projection and Ritz-projection \cite{[vth]}.
\begin{definition}
	The discrete \textit{Laplacian} is a map $\Delta _h : \, X_h \, \rightarrow \, X_h $ defined as
	\begin{equation}\label{4e29a}
	(\Delta _h w_1, \, w_2) \, = \, -(\nabla w_1, \, \nabla w_2), \quad \forall \: w_1, \, w_2 \, \in \, X_h.
	\end{equation}
\end{definition}
\begin{definition}
	The $L^2$ projection is a map  $P_h : L^2(\Omega) \rightarrow X_h$ such that
	\begin{equation}\label{e62a}
	( P_h w_1, \,  w_2) \, = \, ( w_1, \,  w_2), \quad \forall w_1 \in L^2(\Omega) \: \mbox{and}  \; \;  \forall w_2 \in X_h.
	\end{equation}
\end{definition}
\begin{definition}
	The Ritz-projection is a map  $R_h : H^1_0(\Omega) \rightarrow X_h$ such that
	\begin{equation}\label{e62}
	(\nabla R_h w_1, \, \nabla w_2) \, = \, (\nabla w_1, \, \nabla w_2), \quad \forall w_1 \in  H^1_0(\Omega) \: \mbox{and}  \; \;  \forall w_2 \in X_h.
	\end{equation}
\end{definition}
\noindent From \cite[Lemma 1.1]{[vth]}, we obtain 
\begin{equation}\label{4e40}
\begin{split}
\|{w-R_hw} \|_{L^2(\Omega)} \, + h \, \|{\nabla (w-R_hw)}\|_{L^2(\Omega)} \, \le& \, Ch^2 \, \|{\Delta w}\|_{L^2(\Omega)}, \quad \forall w \in H^2 \cap H^1_0.\\
\end{split}
\end{equation}
Next, we discretize the time domain $[0,T]$. For this let $N \in \mathbb{N}$ and $0=t_0<t_1<t_2 <...<t_N=T$ be a partition of the time interval $[0,T]$ such that $t_n=T(\frac{n}{N})^{r}$, for $n=0,1,...,N$, where $r \ge 1$ is the mesh grading parameter. Here, step size $\tau_n$ is given by $\tau_n = t_n - t_{n-1}$, $\forall \, n=1,2,...,N$. The step size $\tau_n$ and grid point $t_n$ satisfy following estimates \cite{[hr12], [r11],[Lio_21]}:
 \begin{equation}\label{4e5}
 \begin{split}
  \tau_n \le C T N^{-r} (n-1)^{r-1}, \quad \mbox{for} \; 2 \le n \le N, \\
 \end{split}
 \end{equation}
and \begin{equation}\label{4e5a}
\begin{split}
t_n \le 2^r \, t_{n-1}, \quad \mbox{for} \; 2 \le n \le N. \\
\end{split}
\end{equation}
From the estimates \eqref{4e5} and \eqref{4e5a}, it follows that there is a constant $C_r > 0,$ independent of the step size $\tau_n$ such that
 \begin{equation}\label{4e5b}
\begin{split}
\tau_{n-1} \le \tau_n  \le C_r \, \tau_{n-1}, \quad \mbox{for} \; 2 \le n \le N. \\
\end{split}
\end{equation}
Now, for any function $w$ defined on $[0,T],$ we define $w^n :=w(t_n)$, for all $n= 0, 1,..., N$ and
\begin{equation}\label{4e12}
\begin{split}
\widehat{w}^n := \, \frac{\tau_{n-1} + \tau_{n}}{\tau_{n-1}} \, w^{n-1} - \frac{\tau_{n}}{\tau_{n-1}} \, w^{n-2}, \quad \mbox{for} \; \, 2 \le n \le N.\\
\end{split}
\end{equation}
The following lemma will be useful in the derivation of \textit{a priori} error estimate.
\begin{lemma}\label{4l3}
	Let $\delta \in (0,1) \cup (1,2)$. If $w(\cdot, t) \in C^1[0,T] \cap C^2(0,T]$ and $\|\partial^q_t w(\cdot,t)\|_{1} \, \le \,  C \, (1 + t^{\delta - q})$ for $q=0,1,2,3$, then
	\begin{equation}\label{4e13}
	| w^n - \widehat{w}^n | \, \le \, C \, N^{-min \left\lbrace 2, \, r \delta \right\rbrace }, \quad \mbox{for} \; \, 2 \le n \le N.\\
	\end{equation}	
\end{lemma}
\begin{proof}
	Proof of this lemma follows from application of Taylor's theorem, \eqref{4e5} and \eqref{4e5a}.
\end{proof}
Now, for a function $w \in C[0,T] \cap C^3(0,T],$ the L1-approximation to  $^{c}_{0}{D}^{\beta}_{t_n} w$ on the graded mesh \cite{[Ch1hyp],[r02],[r11]} can be  given as:
\begin{equation}\label{4e9}
\begin{split}
^{c}_{0}{D}^{\beta}_{t_n} w \, \approx& \, D^{\beta} _{N} w^n
 := \, d_{n,1} \, w^n - d_{n,n} \, w^0 + \, \sum_{k=1}^{n-1} (d_{n,k+1} - d_{n,k}) \, w^{n-k}, \\
\end{split}
\end{equation}
where
\begin{equation}\label{4e10}
\begin{split}
d_{n,k} \, = \, \frac{(t_n - t_{n-k})^{1- \beta} - (t_n - t_{n-k+1})^{1- \beta}}{\Gamma {(2-\beta)} \, \tau_{n-k+1}}, \quad \mbox{for} \; \, 1 \le k \le n \le N.\\
\end{split}
\end{equation}
The following lemma gives the amount of error in the discretization of Caputo fractional derivative using $L1$ scheme on graded mesh.
\begin{lemma}\label{4l2}
	\cite{[Ch1hyp]} Let $\delta \in (0,1) \cup (1,2)$ and $\|\partial^q_t w(x,t)\|_{1} \, \le \,  C \, (1 + t^{\delta - q})$ for $q=0,1,2,3.$ Then there exists a constant $C>0$ such that
	\begin{equation}\label{4e11}
	\| ^{c}_{0}{D}^{\beta}_{t_n} w - D^{\beta} _{N} w^n \|_{1} \, \le \, C \, t^{-\beta}_n \, N^{-min \left\lbrace 2- \beta, \, r \delta \right\rbrace }, \quad \forall n=1, 2, ..., N.
	\end{equation}	
\end{lemma}
\bigskip
Let $U^n$, ${\overline{U}}^n$, $V^n$ denote the approximate value of $u$, $\overline{u}$ and $v$ at $t_n$, respectively, then for $n \ge 2,$ the linearized fully-discrete scheme for \eqref{4e3} is as follows: Find ${\overline{U}}^n,$ $V^n$ $\in X_h$ such that for all $\omega \in X_h,$
\begin{subequations}\label{4e18}
\begin{align}
  \label{4e18:a}
    \big(D^{\beta} _{N} V^n, \omega \big)  + a\big(l(\widehat{U}^n)\big) \, (\nabla {\overline{U}}^n, \nabla \omega) \, =& \, (f^n, \omega)  +  t_n \, a\big(l(\widehat{U}^n)\big) \, (\Delta u_1, \omega), \\
  \label{4e18:b}
    \big(D^{\beta} _{N} {\overline{U}}^n, \omega \big) - (V^n, \omega) \, =& \, 0,\\
  \label{4e18:c}
    U^0 = R_h u_0 \; \mbox{and} \; V^0 =& \, 0,
\end{align}
\end{subequations}
and  $U^1 \, = \, U^0 + \tau_1 \, P_hu_1.$ Note that one can find ${\overline{U}}^1$ and $V^1$ from the relations ${\overline{U}}^1 \, \, = \, U^1 \, - \, t_1 \ P_hu_1$ and $V^1 = \, D^{\beta} _{N} {\overline{U}}^1.$\\
In the following, we write problem \eqref{4e18} in matrix form. Using the definition of $D^{\beta} _{N}$ and putting $\omega=\phi_i$ ($1 \le i \le M$) in \eqref{4e18}, we have
\begin{equation}\label{4e19}
\begin{split}
&\big(d_{n,1} \, V^n, \, \phi_i \big) \, + \, \Big(-d_{n,n} \, V^0 + \sum_{k=1}^{n-1} \, (d_{n,k+1}-d_{n,k}) \, V^{n-k}, \, \phi_i \Big) \, \\
& \qquad \qquad + \, a\big(l(\widehat{U}^n)\big) \, (\nabla {\overline{U}}^n, \, \nabla \phi_i) \, = \, (f^n, \, \phi_i) \, + \, t_n \, a\big(l(\widehat{U}^n)\big) \, (\Delta u_1, \, \phi_i), \\
&\big(d_{n,1} \, {\overline{U}}^n, \, \phi_i \big) + \Big(-d_{n,n} \, {\overline{U}}^0 + \sum_{k=1}^{n-1} \, (d_{n,k+1}-d_{n,k}) \, {\overline{U}}^{n-k} , \, \phi_i \Big) - (V^n, \, \phi_i) \, = \, 0.\\
\end{split}
\end{equation}
Since ${\overline{U}}^n,$ $V^n \in X_h,$  we can write
\begin{equation}\label{4e20}
\begin{split}
   {\overline{U}}^n = \sum_{j=1}^{M} \, \alpha_j^n \, \phi_j \quad \mbox{and} \quad V^n = \sum_{j=1}^{M} \, \beta_j^n \, \phi_j, \quad \mbox{for some } \alpha_j^n, \, \beta_j^n \in \mathbb{R}.
\end{split}
\end{equation}
On substituting these values of ${\overline{U}}^n$ and $V^n$ in \eqref{4e19}, we can obtain the following matrix system.
\begin{equation}\label{4e21}
\begin{split}
  B \, \beta^n + \frac{1}{d_{n,1}} \, a\big(l(\widehat{U}^n)\big) \, A \, \alpha^n + \frac{1}{d_{n,1}} \, B \, G^n = \frac{1}{d_{n,1}} \, F^n + \frac{t_n}{d_{n,1}} \, a\big(l(\widehat{U}^n)\big) \, E,
\end{split}
\end{equation}
and
\begin{equation}\label{4e22}
\begin{split}
d_{n,1} \, B \, \alpha^n + B \, H^n = B \, \beta^n,
\end{split}
\end{equation}
where $B$ is the mass matrix, $A$ is the stiffness matrix, $\beta^n = [\beta_1^n \;\; \beta_2^n \;\; \dots \;\; \beta_M^n]^T,$ \ $\alpha^n = [\alpha_1^n \;\; \alpha_2^n \;\; \dots \;\; \alpha_M^n]^T$ and the elements of vectors $G^n = (G^n_{i1})_{1 \le i \le M},$ $H^n = (H^n_{i1})_{1 \le i \le M},$ $F^n = (F^n_{i1})_{1 \le i \le M},$ $E = (E_{i1})_{1 \le i \le M}$ are given as follows: for each $i=1,2,...,M,$
\begin{equation*}
\begin{split}
   &G^n_{i1} = \Big(-d_{n,n} \, V^0 + \sum_{k=1}^{n-1} \, (d_{n,k+1}-d_{n,k}) \, V^{n-k}, \, \phi_i \Big) \\
   &H^n_{i1} = \Big(-d_{n,n} \, {\overline{U}}^0 + \sum_{k=1}^{n-1} \, (d_{n,k+1}-d_{n,k}) \, {\overline{U}}^{n-k} , \, \phi_i \Big)\\
   &F^n_{i1} = (f^n, \, \phi_i)  \quad \mbox{and} \quad E_{i1} = (\Delta u_1, \, \phi_i).
\end{split}
\end{equation*}
By putting the value of $B \, \beta^n$ from \eqref{4e22} into \eqref{4e21}, we get
\begin{equation}\label{4e23a}
\begin{split}
d_{n,1} \, B \, \alpha^n + B \, H^n + \frac{1}{d_{n,1}} \, a\big(l(\widehat{U}^n)\big) \, A \, \alpha^n + \frac{1}{d_{n,1}} \, B \, G^n = \frac{1}{d_{n,1}} \, F^n + \frac{t_n}{d_{n,1}} \, a\big(l(\widehat{U}^n)\big) \, E,
\end{split}
\end{equation}
Therefore,\\
\begin{equation}\label{4e23b}
\begin{split}
\alpha^n = \Big(d_{n,1} \, B  + \frac{a\big(l(\widehat{U}^n)\big)}{d_{n,1}} \, A\Big)^{-1} \Big( \frac{1}{d_{n,1}} \, F^n + \frac{t_n}{d_{n,1}} \, a\big(l(\widehat{U}^n)\big) \, E - \frac{1}{d_{n,1}} \, B G^n - B H^n \Big),
\end{split}
\end{equation}
and
\begin{equation}\label{4e23c}
\begin{split}
\beta^n = B^{-1} \, \big( d_{n,1} \, B \, \alpha^n + B \, H^n \big) = d_{n,1} \, \alpha^n +H^n.
\end{split}
\end{equation}

\section{\textit{A priori} bound}
In this section, we derive \textit{a priori} bound for fully-discrete solution. In this regard, first we state the coercivity property of the $L1$ scheme in the following lemma.

\begin{lemma}\cite{[Ch1hyp],[ach1]} \label{4l6}
	Let a function $w^n = w( \cdot , \, t_n)$ be in $L^2(\Omega)$, for $n = 0, 1, . . . , N$. Then, one has
	\begin{equation}\label{4e29}
	\left( {D}^{\beta}_{N} w^n, \, w^n \right) \, \ge \, \frac{1}{2} \, {D}^{\beta}_{N} \|w^n\|^2, \quad \forall n=1, 2, ..., N.\\
	\end{equation}
\end{lemma}
\bigskip
Also, for each $n = 1, 2, \dots, N,$ we define the coefficients $Q^{(n)}_{n-i}$ as follows \cite{[r15],[r5]}:
\begin{equation}\label{4e24}
\begin{split}
Q^{(n)}_{n-i} :=  \left\{\begin{array}{l} \frac{1}{d_{i,0}} \sum_{k=i+1}^{n} \left( d_{k, k-i-1} - d_{k, k-i} \right) Q^{(n)}_{n-k}, \; \, \mbox{if} \; i=1, 2,...,(n-1), \\
\\
\frac{1}{d_{n,0}}, \quad \mbox{if} \; \, i=n. \\
\end{array}\right.
\end{split}
\end{equation}	

\begin{lemma}\label{4l4}
  \cite{[r15],[r5]} For a constant $\gamma \in (0,1),$ the coefficients $Q^{(n)}_{n-i}$ defined in \eqref{4e24} satisfy
  \begin{equation}\label{4e25}
  \begin{split}
     \sum_{i=1}^{n} Q^{(n)}_{n-i} \, i^{r(\gamma - \beta)} \, \le \, \frac{\Gamma {(1+ \gamma - \beta)}}{\Gamma {(1+ \gamma)}} \, T^{\beta} \, \Big(\frac{t_n}{T}\Big)^{\gamma} \, N^{{r(\gamma - \beta)}}, \quad \mbox{for} \;\; 1 \le n \le N.
  \end{split}
  \end{equation}
\end{lemma}
The following discrete fractional Gr$\ddot{{o}}$nwall inequality will be useful in the derivation of \textit{a priori} bound and \textit{a priori} error estimate for the fully-discrete solution.
\begin{lemma}\label{4l5}
	\cite{[sfor1]} Let $(g^n)^{N}_{n=1}$ and $(\lambda_i)^{N-1}_{i=0}$ be given nonnegative sequences. Assume that there exists a constant $\Lambda$ such that $ \Lambda \ge \sum_{i=0}^{N-1} \lambda_i$ and the maximum step size satisfies
	\begin{equation}\label{4e26}
	\begin{split}
	  \max_{1 \le n \le N} \, \tau_n \, \le \, \frac{1}{\sqrt[\beta]{4 \, \Gamma (2- \beta) \, \Lambda}}.
	\end{split}
	\end{equation}
	Then, for any nonnegative sequences $(\chi^k)^{N}_{k=0}$ and $(\psi^k)^{N}_{k=0}$ satisfying
	\begin{equation}\label{4e27}
	\begin{split}
	  D^{\beta} _{N} \big[ (\chi^n)^2 + (\psi^n)^2 \big] \, \le \, \sum_{k=1}^{N} \lambda_{n-k} \, \big( \chi^k + \psi^k \big)^2 \, + \, \big( \chi^n + \psi^n \big) \, g^n, \quad 1 \le n \le N,
	\end{split}
	\end{equation}
	it holds that
	\begin{equation}\label{4e28}
	\begin{split}
	  \chi^n + \psi^n \, \le \, 4 \, E_{\beta} \big(4 \, \Lambda \, t_n^{ \beta}\big) \: \Big( \chi^0 \, + \, \psi^0 \, + \, \max_{1 \le k \le n} \, \sum_{i=1}^{k} Q^{(k)}_{k-i} \, g^i \Big),  \quad  \mbox{for} \;\; 1 \le n \le N,
	\end{split}
	\end{equation}
	where $E_{\beta} (z) \, = \, \sum_{k=0}^{\infty} \, \frac{z^k}{\Gamma {(1 + k \beta)}}$ is the Mittag-Leffler function.
\end{lemma}
Now we derive the \textit{a priori} bound for the fully-discrete solution.
\begin{theorem}\label{4th1}
  For each $n=2,3, \dots, N$, the fully-discrete solutions ${\overline{U}}^n$ and $V^n$ satisfy
  \begin{equation}\label{4e30}
  \begin{split}
    \| V^n \| + \| \nabla {\overline{U}}^n \| \le C \, \big(1 + \| V^0 \| + \| \nabla {\overline{U}}^0 \|\big).
  \end{split}
  \end{equation}	
\end{theorem}
\begin{proof}
  We choose $\omega=V^n$ in \eqref{4e18:a} and $\omega= - a\big(l(\widehat{U}^n)\big) \, \Delta_h {\overline{U}}^n$ in \eqref{4e18:b} to get
  \begin{equation}\label{4e31}
  	\begin{split}
  	  \big(D^{\beta} _{N} V^n, \, V^n \big) \, + \, a\big(l(\widehat{U}^n)\big) \, (\nabla {\overline{U}}^n, \, \nabla V^n) \, =& \, (f^n, \, V^n) \, + \, t_n \, a\big(l(\widehat{U}^n)\big) \, (\Delta u_1, \, V^n),\\
  	\end{split}
  \end{equation}
  and
  \begin{equation}\label{4e32}
  \begin{split}
  	- a\big(l(\widehat{U}^n)\big) \, \big(D^{\beta} _{N} {\overline{U}}^n, \, \Delta_h {\overline{U}}^n \big) +  a\big(l(\widehat{U}^n)\big) \, (V^n, \, \Delta_h {\overline{U}}^n) \, =& \, 0.
  	\end{split}
  \end{equation}
  Using the definition of $\Delta_h$ in \eqref{4e32}, we can get
  \begin{equation}\label{4e33}
  \begin{split}
     a\big(l(\widehat{U}^n)\big) \, \big( \nabla D^{\beta} _{N} {\overline{U}}^n, \, \nabla {\overline{U}}^n \big) -  a\big(l(\widehat{U}^n)\big) \, (\nabla V^n, \, \nabla {\overline{U}}^n) \, =& \, 0.
  \end{split}
  \end{equation}
  On adding \eqref{4e31} and \eqref{4e33}, we have
  \begin{equation}\label{4e34}
  \begin{split}
  \big(D^{\beta} _{N} V^n, \, V^n \big) \, +  a\big(l(\widehat{U}^n)\big) \, \big( \nabla D^{\beta} _{N} {\overline{U}}^n, \, \nabla {\overline{U}}^n \big) \, =& \, (f^n, \, V^n) \, + \, t_n \, a\big(l(\widehat{U}^n)\big) \, (\Delta u_1, \, V^n).\\
  \end{split}
  \end{equation}
  Using Lemma \ref{4l6}, bound of $a,$ Cauchy-Schwarz inequality and the fact that $t_n \le T$ in \eqref{4e34}, we can arrive at
  \begin{equation}\label{4e35}
  \begin{split}
  D^{\beta} _{N} \| V^n \|^2  \, + \, m_1 \, D^{\beta} _{N} \| { \nabla \overline{U}}^n \|^2 \, \le& \, \big(\| f^n \| \, + \, T \, m_2 \, \| \Delta u_1 \| \big) \, \| V^n \|.\\
  \end{split}
  \end{equation}
  Therefore,
  \begin{equation}\label{4e36}
  \begin{split}
  D^{\beta} _{N} \big( \| V^n \|^2  \, + \, m_1 \, \| { \nabla \overline{U}}^n \|^2 \big) \, \le& \, \big(\| f^n \| \, + \, T \, m_2 \, \| \Delta u_1 \| \big) \, \big( \| V^n \| + \sqrt{m_1} \, \| { \nabla \overline{U}}^n \| \big).\\
  \end{split}
  \end{equation}
  An application of Lemma \ref{4l5} in \eqref{4e36} $\Big(\mbox {with}$ $\chi^n = \| V^n \|,$ $\psi^n = \sqrt{m_1} \, \| { \nabla \overline{U}}^n\|,$ $\lambda_i = 0,$ for $\forall i = 0, 1, \dots, N-1,$ $g^n = \| f^n \| \, + \, T \, m_2 \, \| \Delta u_1 \|$ and taking $\Lambda = 1 \Big)$ gives
  \begin{equation}\label{4e37}
  \begin{split}
  \| V^n \|  \, + \, \sqrt{m_1} \, \| { \nabla \overline{U}}^n \|& \, \le \, 4 \, E_{\beta} \big( 4 \, t_n^{ \beta} \big) \\ &\Big[ \|V^0\| + \sqrt{m_1} \, \| {\nabla \overline{U}}^0 \| + \max_{1 \le k \le n} \big( \|f^k\| + T m_2 \|\Delta u_1 \| \big) \sum_{i=1}^{k} Q^{(k)}_{k-i} \Big]. \\
  \end{split}
  \end{equation}
 Using Hypothesis H3 and Lemma \ref{4l4} (with $\gamma = \beta$) in \eqref{4e37}, one can obtain
 \begin{equation}\label{4e38}
 \begin{split}
 \| V^n \|  \, + \, \sqrt{m_1} \, \| { \nabla \overline{U}}^n \|& \, \le  \, C \, \big( 1 + \|V^0\| +  \| {\nabla \overline{U}}^0 \| \big),
 \end{split}
 \end{equation}
 where $C = 4 E_{\beta} \big( 4 \, t_n^{ \beta} \big) \, \max \big\{1, \, \sqrt{m_1}, \, \frac{T^{\beta}}{\Gamma {( 1 + \beta)}} \big\}.$\\
 Therefore,
 \begin{equation}\label{4e38a}
 \begin{split}
 \| V^n \|  \, + \| { \nabla \overline{U}}^n \|& \, \le  \, C \, \big( 1 + \|V^0\| +  \| {\nabla \overline{U}}^0 \| \big).
 \end{split}
 \end{equation}
 This completes the proof.
\end{proof}
\begin{remark}\label{4rmk2}
	Using \eqref{4e30}, we can find a bound for $\| \nabla U^n \|$ as follows:
	\begin{equation}\label{4e38b}
	\begin{split}
	\| \nabla U^n \| \, \le \, \|  { \nabla \overline{U}}^n \| + T \| \nabla u_1 \| \, \le  \, C.
	\end{split}
    \end{equation}
\end{remark}
\begin{remark}\label{4rmk3}
	 The existence-uniqueness result for the fully-discrete solution $\overline{U}^n$ and $ V^n$ of \eqref{4e18} follows, in similar lines as given in \cite{[sv_h],[sp1]}, by using \textit{a priori} bound and a consequence of the Brouwer fixed point theorem \cite{[vth]}.
\end{remark}

\section{\textit{A priori} error estimate}
For the derivation of error estimate, we need following regularity assumptions.
\begin{equation}\label{4e39}
\begin{split}
{^{c}_{0}{D}^{\alpha}_{t_n} u} \in {L^{\infty}(0, T; {H^2(\Omega)})} \quad \mbox{and} \quad u \in {L^{\infty}(0, T; {H^1_0(\Omega) \cap H^2(\Omega)})}.
\end{split}
\end{equation}
Similar to \cite{[sfor1]}, in this work work, we also assume that
\begin{equation}\label{4e39a}
\begin{split}
\|\partial^q_t u\|_{1} \, \le \,  C \, (1 + t^{\alpha - q}) \quad \mbox{and} \quad \|\partial^q_t v\|_{1} \, \le \,  C \, (1 + t^{\beta - q}), \quad \mbox{for} \: q=0,1,2,3.
\end{split}
\end{equation}
Now, from the definition of $\overline{u}$ and assumption \eqref{4e39a}, it follows that $\|\partial^q_t \overline{u}\|_{1} \, \le \,  C \, (1 + t^{\alpha - q}),$ for $q=0,1,2,3.$\\
Next, we use the Ritz-projection $R_h$ to split the error as follows:
\begin{equation}\label{4e41}
\begin{split}
\overline{u}^n - \overline{U}^n = \overline{u}^n - R_h\overline{u}^n + R_h\overline{u}^n - \overline{U}^n = \rho ^n + \theta ^n,\\
v^n - V^n = v^n - R_hv^n + R_hv^n - V^n = \xi ^n + \eta ^n,
\end{split}
\end{equation}
where \  $ \rho ^n := \overline{u}^n - R_h\overline{u}^n,$  \ $\theta ^n := R_h\overline{u}^n - \overline{U}^n,$ \ $\xi^n := v^n - R_hv^n$ \ and \ $\eta^n := R_hv^n - V^n.$ \\
\begin{theorem}\label{4th2}
	For $2 \le n \le N$, let $(\overline{u}^n, v^n)$ and $(\overline{U}^n, V^n)$ be the solution of \eqref{4e4} and \eqref{4e18} respectively, then
	\begin{equation}\label{4e42}
	  \| \nabla \overline{u}^n - \nabla \overline{U}^n \| + \| v^n - V^n \| \le C \, \big(h + N^{-\min \left\lbrace 2 - \beta, \, r \beta \right\rbrace }\big).
	\end{equation}
\end{theorem}
\begin{proof}
	From the fully-discrete formulation \eqref{4e18}, we can get
   \begin{equation}\label{4e43}
   	\begin{split}
    	&\big(D^{\beta} _{N} \, \eta^n, \, \omega \big) \, + \, a\big(l(\widehat{U}^n)\big) \, (\nabla \theta^n, \, \nabla \omega) \\
    	&= \, \big(D^{\beta} _{N} R_h v^n, \, \omega \big) \, - \, \big(D^{\beta} _{N} V^n, \, \omega \big) \, + \, a\big(l(\widehat{U}^n)\big) \, (\nabla R_h \overline{u}^n, \, \nabla \omega) \,  - \, a\big(l(\widehat{U}^n)\big) \, (\overline{U}^n, \, \nabla \omega)\\
    	&= \, \big(D^{\beta} _{N} R_h v^n, \, \omega \big) \, + \, a\big(l(\widehat{U}^n)\big) \, (\nabla \overline{u}^n, \, \nabla \omega) \, - \, (f^n, \, \omega) \, - \, t_n \, a\big(l(\widehat{U}^n)\big) \, (\Delta u_1, \, \nabla \omega)\\
    	&= \, \big(D^{\beta} _{N} R_h v^n, \, \omega \big) \, + \, a\big(l(\widehat{U}^n)\big) \, (\nabla \overline{u}^n, \, \nabla \omega) \, - \,
    	\big(^{c}_{0}{D}^{\beta}_{t_n}v, \omega \big) \, - \, a\big(l(u^n)\big) \, (\nabla \overline{u}^n, \nabla \omega) \\
    	&\qquad + \, t_n \, a\big(l(u^n)\big) \, (\Delta u_1, \omega)\, - \, t_n \, a\big(l(\widehat{U}^n)\big) \, (\Delta u_1, \, \nabla \omega)\\
    	&= \, \big(D^{\beta} _{N} R_h v^n - \, ^{c}_{0}{D}^{\beta}_{t_n}v, \, \omega \big) \, + \, \big\{ a \big(l(\widehat{U}^n)\big) -  a\big(l(u^n)\big) \big\} \, (\nabla \overline{u}^n, \nabla \omega) \\
    	&\qquad + \, t_n \, \big\{ a\big(l(u^n)\big) - a \big(l(\widehat{U}^n)\big) \big\} \, (\Delta u_1, \,  \omega),\\
   	\end{split}
   \end{equation}
   and
   \begin{equation}\label{4e44}
   \begin{split}
     \big(D^{\beta} _{N} \theta^n, \, \omega \big) \, - \, ( \eta^n, \, \omega) &= \, \big(D^{\beta} _{N} R_h \overline{u}^n, \, \omega \big) \, - \, \big(D^{\beta} _{N}  \overline{U}^n, \, \omega \big) \, - \, (R_h v^n, \, \omega) \, + \, (V^n, \, \omega)\\
     &= \, \big(D^{\beta} _{N} R_h \overline{u}^n - \, ^{c}_{0}{D}^{\beta}_{t_n} R_h \overline{u}, \, \omega \big).\\
   \end{split}
   \end{equation}
   Choosing $\omega = \eta^n$ in \eqref{4e43} and $\omega= - a \big(l(\widehat{U}^n)\big) \, \Delta_h \theta^n$ in \eqref{4e44}, we have
   \begin{equation}\label{4e45}
   \begin{split}
     &\big(D^{\beta} _{N} \, \eta^n, \, \eta^n \big) \, + \, a\big(l(\widehat{U}^n)\big) \, (\nabla \theta^n, \, \nabla \eta^n) = \, \big(D^{\beta} _{N} R_h v^n - \, ^{c}_{0}{D}^{\beta}_{t_n}v, \, \eta^n \big) \\
     &+ \, \big\{ a \big(l(\widehat{U}^n)\big) -  a\big(l(u^n)\big) \big\} \, (\nabla \overline{u}^n, \nabla \eta^n) \, + \, t_n \, \big\{ a\big(l(u^n)\big) - a \big(l(\widehat{U}^n)\big) \big\} \, (\Delta u_1, \,  \eta^n),\\
   \end{split}
   \end{equation}
   and
   \begin{equation}\label{4e46}
   \begin{split}
   - a \big(l(\widehat{U}^n)\big) \, \big(D^{\beta} _{N} \theta^n, \, \Delta_h \theta^n \big) \, + \,  &a \big(l(\widehat{U}^n)\big) \, ( \eta^n, \, \Delta_h \theta^n)  \qquad \qquad \qquad \qquad \qquad \qquad \qquad \qquad\\
    & \; = \, - a \big( l(\widehat{U}^n)\big) \, \big(D^{\beta} _{N} R_h \overline{u}^n - \, ^{c}_{0}{D}^{\beta}_{t_n} R_h \overline{u}, \, \Delta_h \theta^n \big).\\
   \end{split}
   \end{equation}
   Using the definition of $\Delta_h,$ \eqref{4e46} can be written as
   \begin{equation}\label{4e47}
   \begin{split}
    a \big(l(\widehat{U}^n)\big) \, \big( \nabla D^{\beta} _{N} \theta^n, \, \nabla \theta^n \big) \, - \,  &a \big(l(\widehat{U}^n)\big) \, ( \nabla \eta^n, \, \nabla \theta^n)  \qquad \qquad \qquad \qquad \qquad \qquad \qquad \qquad\\
   & \; = \,  a \big( l(\widehat{U}^n)\big) \, \big( \nabla D^{\beta} _{N} R_h \overline{u}^n - \, \nabla \, ^{c}_{0}{D}^{\beta}_{t_n} R_h \overline{u}, \, \nabla \theta^n \big).\\
   \end{split}
   \end{equation}
   On adding \eqref{4e45} and \eqref{4e47}, we have
   \begin{equation}\label{4e48}
   \begin{split}
    &\big(D^{\beta} _{N} \, \eta^n, \, \eta^n \big) \, + \, a \big(l(\widehat{U}^n)\big) \, \big( \nabla D^{\beta} _{N} \theta^n, \, \nabla \theta^n \big) = \, \big(D^{\beta} _{N} R_h v^n - \, ^{c}_{0}{D}^{\beta}_{t_n}v, \, \eta^n \big) \\
    &+ \, \big\{ a\big(l(u^n)\big) - a \big(l(\widehat{U}^n)\big) \big\} \, (\Delta \overline{u}^n, \, \eta^n) \, + \, t_n \, \big\{ a\big(l(u^n)\big) - a \big(l(\widehat{U}^n)\big) \big\} \, (\Delta u_1, \, \eta^n)\\
    &+ a \big( l(\widehat{U}^n)\big) \, \big( \nabla D^{\beta} _{N} \overline{u}^n - \, \nabla \, ^{c}_{0}{D}^{\beta}_{t_n} \overline{u}, \, \nabla \theta^n \big).
   \end{split}
   \end{equation}
   Using the Cauchy-Schwarz inequality and bound of $a$ in \eqref{4e48}, we get
   \begin{equation}\label{4e49}
   \begin{split}
    &\big(D^{\beta} _{N} \, \eta^n, \, \eta^n \big) \, + \, m_1 \, \big( \nabla D^{\beta} _{N} \theta^n, \, \nabla \theta^n \big) \le \, \big\| D^{\beta} _{N} R_h v^n - \, ^{c}_{0}{D}^{\beta}_{t_n}v \big\| \, \| \eta^n \| \\
    &+ \, \big| a\big(l(u^n)\big) - a \big(l(\widehat{U}^n)\big) \big| \, \| \Delta \overline{u}^n \| \, \| \eta^n \| \, + \, T \, \big| a\big(l(u^n)\big) - a \big(l(\widehat{U}^n)\big) \big| \, \| \Delta u_1 \| \, \| \eta^n \|\\
    &+ m_2 \, \big\| \nabla D^{\beta} _{N}  \overline{u}^n - \, \nabla \, ^{c}_{0}{D}^{\beta}_{t_n}  \overline{u} \big\| \, \| \nabla \theta^n \|.
   \end{split}
   \end{equation}
   Now,
   \begin{equation}\label{4e50}
   \begin{split}
      \big\| D^{\beta} _{N} R_h v^n - \, ^{c}_{0}{D}^{\beta}_{t_n}v \big\| \, \le& \, \big\| D^{\beta} _{N} R_h v^n - \, ^{c}_{0}{D}^{\beta}_{t_n} R_h v \big\| + \big\| \, ^{c}_{0}{D}^{\beta}_{t_n} R_h v - \, ^{c}_{0}{D}^{\beta}_{t_n}v \big\|\\
      \le& \, C \, \big( h^2 + t_n^{- \beta} \, N^{-\min \left\lbrace 2 - \beta, \, r \beta \right\rbrace } \big),
   \end{split}
   \end{equation}
   and
   \begin{equation}\label{4e51}
   \begin{split}
   \big\| \nabla D^{\beta} _{N}  \overline{u}^n - \, \nabla \, ^{c}_{0}{D}^{\beta}_{t_n} \overline{u} \big\|
   \le \, C \,  t_n^{- \beta} \, N^{-\min \left\lbrace 2 - \beta, \, r \alpha \right\rbrace } \,
   \le& \, C \,  t_n^{- \beta} \, N^{-\min \left\lbrace 2 - \beta, \, r \beta \right\rbrace },
   \end{split}
   \end{equation}
   where we have used Lemma \ref{4l2} and \eqref{4e40}.\\
   Also,
   \begin{equation}\label{4e52}
   \begin{split}
     \big| a\big(l(u^n)\big) - a \big(l(\widehat{U}^n)\big) \big| \, \le& \, L \, \big| l(u^n) -  l(\widehat{U}^n) \big| \, \le \, L \, \big| \|\nabla u^n \|^2 - \|\nabla \widehat{U}^n \|^2 \big| \\
     \le& \, L  \, \big| \|\nabla u^n \| - \|\nabla \widehat{U}^n \| \big| \, \big( \|\nabla u^n \| + \|\nabla \widehat{U}^n \| \big) \\
     \le& \, C \, L  \,  \|\nabla u^n  - \nabla \widehat{U}^n \| \\
     \le& \, C L \, \big(\|\nabla u^n - \nabla \widehat{u}^n\| + \|\nabla \widehat{u}^n - \nabla \widehat{U}^n\| \big).\\
   \end{split}
   \end{equation}
   From Lemma \ref{4l3}, one can get
   \begin{equation}\label{4e53}
   \begin{split}
      \|\nabla u^n - \nabla \widehat{u}^n\| \, \le \, C \, N^{-min \left\lbrace 2, \, r \alpha \right\rbrace } \, \le \, C \, N^{-min \left\lbrace 2, \, r \beta \right\rbrace },
   \end{split}
   \end{equation}
   and from \eqref{4e5b} and \eqref{4e40}, we have
   \begin{equation}\label{4e54}
   \begin{split}
     \|\nabla \widehat{u}^n - \nabla \widehat{U}^n\| \,
     =& \, \Big\| \Big( 1 + \frac{\tau_n}{\tau_{n-1}}\Big) \, \nabla( u^{n-1} -  U^{n-1}) - \frac{\tau_n}{\tau_{n-1}} \, \nabla( u^{n-2} - U^{n-2}) \Big\| \\
     =& \, \Big\| \Big( 1 + \frac{\tau_n}{\tau_{n-1}}\Big) \, \nabla(\rho^{n-1} + \theta^{n-1}) - \frac{\tau_n}{\tau_{n-1}} \, \nabla(\rho^{n-2} + \theta^{n-2}) \Big\| \\
     \le& \, \Big( 1 + \frac{\tau_n}{\tau_{n-1}}\Big) \, \big(\|\nabla \rho^{n-1} \| +  \|\nabla \theta^{n-1} \| \big) + \frac{\tau_n}{\tau_{n-1}} \, \big( \|\nabla \rho^{n-2} \| +  \|\nabla \theta^{n-2} \| \big)\\
     \le& \, C (1+ C_r) h + (1+ C_r) \|\nabla \theta^{n-1} \| + C C_r h + C_r \|\nabla \theta^{n-2} \| \\
     \le& C (1+ 2 C_r) h + (1+ C_r) \|\nabla \theta^{n-1} \| + C_r \|\nabla \theta^{n-2} \|. \\
    \end{split}
   \end{equation}
   Using \eqref{4e53} and \eqref{4e54} in \eqref{4e52}, we can arrive at
   \begin{equation}\label{4e55}
   \begin{split}
   \big| a\big(l(u^n)\big) - a \big(l(\widehat{U}^n)\big) \big| \, \le& \, C L N^{-min \left\lbrace 2, \, r \beta \right\rbrace } \, + \,  C L (1+ 2 C_r) \, h \, + \, C L (1+ C_r) \, \|\nabla \theta^{n-1} \| \\
   &+ C L C_r \, \|\nabla \theta^{n-2} \|.\\
   \end{split}
   \end{equation}
   Substituting the values from  \eqref{4e50}, \eqref{4e51}, \eqref{4e55} into \eqref{4e49} together with \eqref{4e39} and Hypothesis H3, we get
   \begin{equation}\label{4e56}
   \begin{split}
     &\big(D^{\beta} _{N} \, \eta^n, \, \eta^n \big) \, + \, m_1 \, \big( \nabla D^{\beta} _{N} \theta^n, \, \nabla \theta^n \big) \le \, C \, \big( h^2 + t_n^{- \beta} \, N^{-\min \left\lbrace 2 - \beta, \, r \beta \right\rbrace } \big)  \, \| \eta^n \| \\
      &+ \, \Big(C L N^{-min \left\lbrace 2, \, r \beta \right\rbrace } \,
     + \,  C L (1+ 2 C_r) \, h \, + \, C L (1+ C_r) \, \|\nabla \theta^{n-1} \| + C L C_r \, \|\nabla \theta^{n-2} \|\Big) \\
     &\qquad C (1 + T) \, \| \eta^n \| \, + \, C \, m_2 \, t_n^{- \beta} \, N^{-\min \left\lbrace 2 - \beta, \, r \beta \right\rbrace } \, \| \nabla \theta^n \|.
   \end{split}
   \end{equation}
   Using Lemma \ref{4l6} and the fact that $C \, (h^2 + h + N^{-min \left\lbrace 2, \, r \beta \right\rbrace } ) \, \sqrt{m_1} \, \| \nabla \theta^n \| \, \ge \, 0$ in \eqref{4e56}, we can obtain
   \begin{equation}\label{4e57}
   \begin{split}
   &D^{\beta} _{N} \big( \| \eta^n \|^2  \, + \, m_1 \, \| \nabla \theta^n \|^2 \big) \, \le \, C_1 \, \big(h^2 + h + N^{-min \left\lbrace 2, \, r \beta \right\rbrace } + t_n^{- \beta} \, N^{-\min \left\lbrace 2 - \beta, \, r \beta \right\rbrace } \big)  \, \| \eta^n \| \qquad \qquad \qquad \qquad\\
   &+ \, \frac{C L}{\sqrt{m_1}} (1+ C_r) \, (1 + T ) \, \sqrt{m_1} \, \| \nabla \theta^{n-1} \| \, \| \eta^n \| + \, C \, \big( h^2 + h + N^{-min \left\lbrace 2, \, r \beta \right\rbrace } \big) \, \sqrt{m_1} \, \| \nabla \theta^n \| \\
   &+ \,  \frac{C L C_r}{\sqrt{m_1}} \, (1 + T ) \, \sqrt{m_1} \, \| \nabla \theta^{n-2} \| \, \| \eta^n \| + \, \frac{C \, m_2}{\sqrt{m_1}} \, t_n^{- \beta} \, N^{-\min \left\lbrace 2 - \beta, \, r \beta \right\rbrace } \, \sqrt{m_1} \, \| \nabla \theta^n \|,\\
   \end{split}
   \end{equation}
   where $C_1$ is dependent on $L,$ $C_r$  and $T.$ \\
   An application of the inequality $ab \le \frac{a^2}{2} + \frac{b^2}{2}$ in \eqref{4e57} gives
   \begin{equation}\label{4e58}
   \begin{split}
   D^{\beta} _{N} \big( \| \eta^n \|^2  \, + \, m_1 \, \| \nabla \theta^n \|^2 \big) \, \le \, \frac{C_3}{2} \, m_1 \, \| \nabla \theta^{n-1} \|^2 \, + \,  \| \eta^n \|^2 \, +  \,  \frac{C_4}{2} \, m_1 \, \| \nabla \theta^{n-2} \|^2& \\
   +\, C_2 \, \big(h^2 +  h + N^{-min \left\lbrace 2, \, r \beta \right\rbrace } + t_n^{- \beta} \, N^{-\min \left\lbrace 2 - \beta, \, r \beta \right\rbrace } \big)  \, \big( \| \eta^n \| + \, \sqrt{m_1} \, \| \nabla \theta^n \| \big),& \\
   \end{split}
   \end{equation}
   where constant $C_2$ is depending on $C_1,$ $m_1,$ $m_2$ and constants $C_3$ and $C_4$ are depending on $L,$ $C_r,$ $T,$ $m_1.$ \\
   As $\Big(\frac{C_3}{2} \, \| \eta^{n-1} \|^2 \, + \, \frac{C_4}{2} \, \| \eta^{n-2} \|^2 \, + \, m_1 \, \| \nabla \theta^n \|^2 \Big) \, \ge \, 0,$ from \eqref{4e58}, we have
   \begin{equation}\label{4e59}
   \begin{split}
   D^{\beta} _{N} \big( \| \eta^n \|^2  \, + \, m_1 \, \| \nabla \theta^n \|^2 \big) &\, \le  \frac{C_3}{2} \, \big(\| \eta^{n-1} \|^2 \, + \, m_1 \, \| \nabla \theta^{n-1} \|^2 \big) \, + \,  \big( \| \eta^n \|^2 \, + \, m_1 \, \| \nabla \theta^n \|^2 \big) \\
   &+  C_2 \, \big( h^2 + h + N^{-min \left\lbrace 2, \, r \beta \right\rbrace } + t_n^{- \beta} \, N^{-\min \left\lbrace 2 - \beta, \, r \beta \right\rbrace } \big) \\
   & \big( \| \eta^n \| + \, \sqrt{m_1} \, \| \nabla \theta^n \| \big) \, +  \,  \frac{C_4}{2} \, \big(\| \eta^{n-2} \|^2 \, + \, m_1 \, \| \nabla \theta^{n-2} \|^2 \big). \\
   \end{split}
   \end{equation}
   Applying Lemma \ref{4l5} in \eqref{4e59} $\Big(\mbox {with}$ $\chi^n = \| \eta^n \|,$ $\psi^n = \sqrt{m_1} \, \| \nabla \theta^n \|,$ $\lambda_0 = 1,$ $\lambda_1 = \frac{C_3}{2},$ $\lambda_2 = \frac{C_4}{2},$ $\lambda_i = 0,$ for $\forall i = 3, 4, \dots, N-1,$ $g^n = C_2 \, \big( h^2 + h + N^{-min \left\lbrace 2, \, r \beta \right\rbrace } + t_n^{- \beta} \, N^{-\min \left\lbrace 2 - \beta, \, r \beta \right\rbrace } \big)$ and taking $\Lambda = \lambda_0 + \lambda_1 + \lambda_2 \Big)$, we can arrive at
   \begin{equation}\label{4e60}
   \begin{split}
   \| \eta^n \| \, + \, \sqrt{m_1} \, \| \nabla \theta^n \| \, \le \, 4  &E_{\beta} \big( 4 \, \Lambda \, t_n^{ \beta} \big) \, \Big[ \| \eta^0 \| + \sqrt{m_1} \, \| \nabla \theta^0 \| + C_2 \big(h^2 + h + N^{-min \left\lbrace 2, \, r \beta \right\rbrace } \big)\\
   & \max_{1 \le k \le n}  \sum_{i=1}^{k} Q^{(k)}_{k-i} \, + C_2 \, N^{-\min \left\lbrace 2 - \beta, \, r \beta \right\rbrace } \, \max_{1 \le k \le n}  \sum_{i=1}^{k} Q^{(k)}_{k-i} \, t_i^{- \beta} \Big]. \\
   \end{split}
   \end{equation}
   Now,
   \begin{equation}\label{4e61}
   \begin{split}
   t_i^{- \beta} \, \le \, T^{-\beta} N^{r \beta} \, i^{-r \beta} \, \le \, T^{-\beta} N^{r \beta} \, i^{ r (l_N -\beta ) }, \quad \mbox{where} \; \: l_N = \frac{1}{\ln N}.
   \end{split}
   \end{equation}
   Using Lemma \ref{4l4} with $\gamma = l_N$ and $\gamma = \beta$, we have
   \begin{equation}\label{4e62}
   \begin{split}
   \sum_{i=1}^{k} Q^{(k)}_{k-i} \, i^{r(l_N - \beta)} \, \le& \, \frac{\Gamma {(1+ l_N - \beta)}}{\Gamma {(1+ l_N)}} \, T^{\beta} \, \Big(\frac{t_n}{T}\Big)^{l_N} \, N^{{r(l_N - \beta)}}\\
   \le&  \, \frac{\Gamma {(1+ l_N - \beta)}}{\Gamma {(1+ l_N)}} \, T^{\beta} \, N^{{r(l_N - \beta)}},
   \end{split}
   \end{equation}
   and
   \begin{equation}\label{4e63}
   \begin{split}
   	\sum_{i=1}^{k} Q^{(k)}_{k-i} \, \le& \, \frac{T^{\beta}}{\Gamma {(1+ \beta)}} \,  \, \Big(\frac{t_n}{T}\Big)^{\beta}  \, \le  \, \frac{T^{\beta}}{\Gamma {(1+ \beta)}}.
   \end{split}
   \end{equation}
  Since $\overline{U}^0 = U^0 = R_h u_0$ and $V^0 = 0,$ we get $\| \nabla \theta^0 \| = 0$ and $\| \nabla \eta^0 \| = 0.$ \\
   Substituting the values from \eqref{4e61}-\eqref{4e63} in \eqref{4e60}, one can get
   \begin{equation}\label{4e64}
   \begin{split}
   \| \eta^n \|  \, + \, \sqrt{m_1} \, \| \nabla \theta^n \| \, \le& \, 4 \, E_{\beta} \big( 4 \, \Lambda \, t_n^{ \beta} \big) \, \Big[ C_2 \, \big(h^2 + h + N^{-min \left\lbrace 2, \, r \beta \right\rbrace } \big) \, \frac{T^{\beta}}{\Gamma {(1+ \beta)}} \qquad\\
   &+ C_2 \, N^{-\min \left\lbrace 2 - \beta, \, r \beta \right\rbrace } \, \frac{\Gamma {(1+ l_N - \beta)}}{\Gamma {(1+ l_N)}} \, T^{-\beta} N^{r \beta} \, T^{\beta} \, N^{{r(l_N - \beta)}} \Big] \\
   \le& \, C \, \big(h^2 + h + N^{-min \left\lbrace 2, \, r \beta \right\rbrace } + N^{-\min \left\lbrace 2 - \beta, \, r \beta \right\rbrace } \big)\\
   \le& \, C \, \big( h  +  N^{-\min \left\lbrace 2 - \beta, \, r \beta \right\rbrace } \big),
   \end{split}
   \end{equation}
   where $C$ is dependent on $C_2,$ $T,$ $l_N.$ \\
   Since $\min \{1, \sqrt{m_1}\} \, \big(\| \eta^n \|  \, + \, \| \nabla \theta^n \|\big) \, \le \, \| \eta^n \|  \, + \, \sqrt{m_1} \, \| \nabla \theta^n \|,$
   \begin{equation}\label{4e65}
   \begin{split}
     \big(\| \eta^n \|  \, + \, \| \nabla \theta^n \|\big) \, \le \, C \, \big( h  +  N^{-\min \left\lbrace 2 - \beta, \, r \beta \right\rbrace } \big).
   \end{split}
   \end{equation}
   Thus, using \eqref{4e40} and \eqref{4e65} together with triangle inequality, we obtain
   \begin{equation}\label{4e66}
   \| \nabla \overline{u}^n - \nabla \overline{U}^n \| + \| v^n - V^n \| \le C \, \big(h + N^{-\min \left\lbrace 2 - \beta, \, r \beta \right\rbrace }\big).
   \end{equation}
   This completes the proof.
\end{proof}
\begin{remark}\label{4rmk5}
	In above work, we have assumed that problem \eqref{4e1} or \eqref{4e4} has a unique solution with sufficient regularity. The proof of the same is yet to be discovered.
\end{remark}
\section{Numerical experiments}
   In order to confirm our theoretical findings, in this section we perform some numerical experiments. In each example, the time interval is taken to be $[0,1]$ and a quasi-uniform partition of $\Omega$ is used with $(M_s+1)$ node points in each spatial direction.
  In order to get an optimal convergence rate, we choose grading parameter $r = \frac{2- \beta}{\beta}.$ To obtain the rate of convergence in temporal direction,
  we set $M_s=N^{2- \beta}.$ Similarly, We set $N=M_s^{\frac{2}{2- \beta}}$ to conclude the order of convergence in spatial direction. \medskip \\
\begin{example}\label{4E1}
	In first example, we take $\Omega = (0, \pi),$ $a(w)=3+\sin w$ and choose $f(x,t)$ in \eqref{4e1} such that the analytical solution of problem \eqref{4e1} is $u(x,t)=(t^3+t^{\alpha}) \sin x$. 	
\end{example}
As shown in Table \ref{4T1}, on graded mesh, $N^{-(2-\beta)}$ convergence rate in temporal direction is obtained in $L^{\infty}\big(H^1_0(\Omega)\big)$ norm for the different values of $\alpha.$
Table \ref{4T2} confirms the order of convergence in the spatial direction in $H^1_0(\Omega)$ norm.
\begin{center}
	\begin{tiny}
		\begin{table}[h!]
			\begin{center}
				\begin{tabular}{|c|c|c|c|c|c|c|}
					\hline
					\multirow{2}{*}{\large $N$} & \multicolumn{2}{c|}{$\alpha = 1.4$} & \multicolumn{2}{c|}{$\alpha = 1.5$} & \multicolumn{2}{c|}{$\alpha = 1.8$}\\
					\cline{2-7}
					& Error & OC & Error & OC & Error & OC  \\
					\hline
					$2^7$ & 7.01E-03 &  1.266809 & 8.63E-03 & 1.226166 & 1.62E-02 & 1.090252 \\
					\cline{1-1}\cline{2-7}
					$2^8$  & 2.91E-03 &  1.279812 & 3.69E-03 & 1.235579 & 7.63E-03 & 1.091875 \\
					\cline{1-1}\cline{2-7}
					$2^9$ & 1.20E-03 & 1.288394 & 1.57E-03 & 1.241687 & 3.58E-03 & 1.093030 \\
					\cline{1-1}\cline{2-7}
					$2^{10}$ &  4.91E-04 & - & 6.63E-04 & - & 1.68E-03 & - \\			
					\hline															
				\end{tabular}
			\end{center}
			\caption {\emph {Error and order of convergence  in $L^{\infty}\big(H^1_0(\Omega)\big)$ norm in temporal direction using graded mesh for Example \ref{4E1}.}}
			\label{4T1}
		\end{table}
	\end{tiny}
\end{center}
\begin{center}
	\begin{tiny}
		\begin{table}[h!]
			\begin{center}
				\begin{tabular}{|c|c|c|c|c|c|c|}
					\hline
					\multirow{2}{*}{\large $M_s$} & \multicolumn{2}{c|}{$\alpha = 1.4$} & \multicolumn{2}{c|}{$\alpha = 1.5$} & \multicolumn{2}{c|}{$\alpha = 1.8$}\\
					\cline{2-7}
					& Error & OC & Error & OC & Error & OC  \\
					\hline
					$2^4$ & 1.43E-01 & 1.005690 & 1.43E-01 & 1.005830 & 1.43E-01 & 1.005599 \\
					\cline{1-1}\cline{2-7}
					$2^5$ & 7.12E-02 & 1.001684 & 7.11E-02 & 1.001614 & 7.11E-02 & 1.001458 \\
					\cline{1-1}\cline{2-7}
					$2^6$  & 3.55E-02 & 1.000449 & 3.55E-02 & 1.000422 & 3.55E-02 & 1.000374 \\
					\cline{1-1}\cline{2-7}
					$2^7$ & 1.78E-02 & - & 1.78E-02 & - & 1.78E-02 & - \\
					\hline															
				\end{tabular}
			\end{center}
			\caption {\emph {Error and order of convergence in $H^1_0(\Omega)$ norm in spatial direction for Example \ref{4E1}.}}
			\label{4T2}
		\end{table}
	\end{tiny}
\end{center}
\begin{example}\label{4E2}
	For second example, we take $\Omega = (0, 1) \times (0, 1),$ $a(w)=3+\sin w$ and choose  $f(x,t)$ in such a way that the analytical solution of PDE in \eqref{4e1} is $u(x,t)=(t^3+t^{\alpha}) (x-x^2)(y-y^2)$.	
\end{example}
From Table \ref{4T3} and Table \ref{4T4}, it can be observed that for different values of $\alpha$, we get optimal order of convergence in temporal as well as spatial directions which is predicted in Theorem \eqref{4th2}.
\begin{center}
	\begin{tiny}
		\begin{table}[h!]
			\begin{center}
				\begin{tabular}{|c|c|c|c|c|c|c|}
					\hline
					\multirow{2}{*}{\large $N$} & \multicolumn{2}{c|}{$\alpha = 1.4$} & \multicolumn{2}{c|}{$\alpha = 1.5$} & \multicolumn{2}{c|}{$\alpha = 1.8$}\\
					\cline{2-7}
					& Error & OC & Error & OC & Error & OC  \\
					\hline
					$2^4$ & 1.00E-02 & 1.323437 & 1.12E-02 & 1.248964 & 1.71E-02 & 1.098823 \\
					\cline{1-1}\cline{2-7}
					$2^5$  & 4.00E-03 & 1.304197 & 4.73E-03 & 1.252916 & 7.97E-03 & 1.107867 \\
					\cline{1-1}\cline{2-7}
					$2^6$  & 1.62E-03 & 1.305113 & 1.99E-03 & 1.249295 & 3.70E-03 & 1.093619 \\
					\cline{1-1}\cline{2-7}
					$2^7$ & 6.56E-04 & - & 8.35E-04 & - & 1.73E-03 & - \\
					\hline															
				\end{tabular}
			\end{center}
			\caption {\emph {Error and order of convergence  in $L^{\infty}\big(H^1_0(\Omega)\big)$ norm in temporal direction using graded mesh for Example \ref{4E2}.}}
			\label{4T3}
		\end{table}
	\end{tiny}
\end{center}
\begin{center}
	\begin{tiny}
		\begin{table}[h!]
			\begin{center}
				\begin{tabular}{|c|c|c|c|c|c|c|}
					\hline
					\multirow{2}{*}{\large $M_s$} & \multicolumn{2}{c|}{$\alpha = 1.4$} & \multicolumn{2}{c|}{$\alpha = 1.5$} & \multicolumn{2}{c|}{$\alpha = 1.8$}\\
					\cline{2-7}
					& Error & OC & Error & OC & Error & OC  \\
					\hline
					$2^3$ & 4.45E-02 & 0.993890 & 4.45E-02 & 0.993887 & 4.45E-02 & 0.993892 \\
					\cline{1-1}\cline{2-7}
					$2^4$ & 2.24E-02 & 0.998464 & 2.24E-02 & 0.998464 & 2.24E-02 & 0.998465 \\
					\cline{1-1}\cline{2-7}
					$2^5$  & 1.12E-02 & 0.999615 & 1.12E-02 & 0.999615 & 1.12E-02 & 0.999615 \\
					\cline{1-1}\cline{2-7}
					$2^6$ & 5.60E-03 & - & 5.60E-03 & - & 5.60E-03 & - \\
					\hline 															
				\end{tabular}
			\end{center}
			\caption {\emph {Error and order of convergence in $H^1_0(\Omega)$ norm in spatial direction for Example \ref{4E2}.}}
			\label{4T4}
		\end{table}
	\end{tiny}
\end{center}
\begin{remark}\label{4rmk4}
	In present work, we have used the $L1$ scheme to discretize a time variable which leads to $O\big(N^{-\min \left\lbrace 2 - \beta, \, r \beta \right\rbrace }\big)$ accuracy in temporal direction. One can also use $L2$-$1_{\sigma}$ approximation \cite{[AAl2]} to discretize the Caputo derivative in order to achieve $O\big(N^{-\min \left\lbrace 2, \, r \beta \right\rbrace }\big)$ accuracy in temporal direction.
\end{remark}






\end{document}